%% file: normal-essential-codimension-l.tex
\documentclass{amsart}

\usepackage{amsmath}
\usepackage{amssymb}
\usepackage{nicefrac}

\usepackage{xcolor}

\usepackage[style=alphabetic,firstinits=true]{biblatex}
\usepackage[centercolon]{mathtools}
\usepackage{thmtools}
\usepackage{enumitem}
\setlist[enumerate]{label={\upshape(\roman*)}}
\usepackage[colorlinks=true, citecolor=red, linkcolor=blue, urlcolor=blue]{hyperref}
\usepackage{cleveref}

\input{loreaux3}

\numberwithin{equation}{section}

\newtheorem{theorem}{Theorem}[section]
\newtheorem{corollary}[theorem]{Corollary}
\newtheorem{lemma}[theorem]{Lemma}
\newtheorem{proposition}[theorem]{Proposition}

\newtheoremstyle{named}{}{}{\itshape}{}{\bfseries}{.}{.5em}{\thmnote{#3}}
\theoremstyle{named}

\theoremstyle{definition}
\newtheorem{remark}[theorem]{Remark}
\newtheorem{definition}[theorem]{Definition}
\newtheorem{example}[theorem]{Example}

\DeclareMathOperator{\rank}{rank}

\DeclareMathOperator{\dist}{dist}
\DeclareMathOperator{\Lim}{Lim^1}
\DeclareMathOperator{\diag}{diag}
\DeclareMathOperator{\trace}{Tr}
\DeclareMathOperator{\spans}{span}
\DeclareMathOperator{\ind}{ind}
\DeclareMathOperator{\spec}{\sigma}
\newcommand{\essspec}{\spec_{\mathrm{ess}}}
\newcommand{\Hil}{\ensuremath{\mathcal{H}}}
\newcommand{\vmid}{\,\middle\vert\,}
\renewcommand{\epsilon}{\varepsilon}

\begin{document}

\title[Restricted diagonalization of finite spectrum normal operators]{Restricted diagonalization of \\ finite spectrum normal operators  \\ and a theorem of Arveson}
\author{Jireh Loreaux}
\email{\href{mailto:jloreau@siue.edu}{jloreau@siue.edu}}
\address{Southern Illinois University Edwardsville \\
  Department of Mathematics and Statistics \\
  Edwardsville, IL, 62026-1653 \\
  USA}

\keywords{Essential codimension, diagonals of projections, diagonals of normal operators}
\subjclass{Primary 47B15, 47A53; Secondary 46C05.}

\begin{abstract}
  Kadison characterized the diagonals of projections and observed the presence of an integer.
  Arveson later recognized this integer as a Fredholm index obstruction applicable to any normal operator with finite spectrum coincident with its essential spectrum whose elements are the vertices of a convex polygon.
  Recently, in joint work with Kaftal, the author linked the Kadison integer to essential codimension of projections.

  This paper provides an analogous link between Arveson's obstruction and essential codimension as well as a new approach to Arveson's theorem which also allows for generalization to any finite spectrum normal operator.
  In fact, we prove that Arveson's theorem is a corollary of a trace invariance property of arbitrary normal operators.
  An essential ingredient is a formulation of Arveson's theorem in terms of diagonalization by a unitary which is a Hilbert--Schmidt perturbation of the identity.
\end{abstract}

\maketitle

\section{Introduction}

A \emph{diagonal} of a bounded linear operator $T \in B(\Hil)$ is a sequence of inner products $\big(\langle Te_n,e_n \rangle\big)$ where $\{e_n\}_{n=1}^{\infty}$ is an orthonormal basis for the Hilbert space $\Hil$.
In other words, a diagonal of $T$ is the diagonal of some matrix representation of $T$ with respect to an orthonormal basis.

In his seminal papers on the Pythagorean Theorem \cite{Kad-2002-PNASU,Kad-2002-PNASUa} Kadison proved the following characterization of diagonals of projections.

\begin{theorem}[Kadison]
  \label{thm:kadison-carpenter-pythagorean}
  A sequence $(d_n)$ is the diagonal of a projection $P$ if and only if it takes values in the unit interval and the quantities
  \begin{equation*}
    a := \sum_{d_n < \nicefrac{1}{2}} d_n \quad\text{and}\quad b := \sum_{d_n \ge \nicefrac{1}{2}} (1-d_n)
  \end{equation*}
  satisfy one of the mutually exclusive conditions
  \begin{enumerate}
  \item $a+b = \infty$;
  \item\label{item:a+b-finite} $a+b < \infty$ and $a-b \in \mathbb{Z}$.
  \end{enumerate}
\end{theorem}

The existence of the integer $a-b$ is not at all obvious and Kadison himself referred to it as ``curious.''
Since Kadison's initial paper, both Arveson \cite[Theorem~3]{Arv-2007-PNASU} and Argerami \cite[Theorem~4.6]{Arg-2015-IEOT} have provided new proofs that $a-b \in \mathbb{Z}$.
Recently, the author and Kaftal further clarified this integer in \cite{KL-2017-IEOT} as the essential codimension between the projection $P$ and a natural diagonal projection associated to $a,b$.
Essential codimension was developed by Brown, Douglas and Fillmore in \cite[Remark~4.9]{BDF-1973-PoaCoOT} (see also \autoref{def:essential-codimension} below) for pairs of projections whose difference is compact.

Arveson also recognized the Kadison integer as the index of a Fredholm operator in \cite{Arv-2007-PNASU}, and referred to it as an ``index obstruction'' to an arbitrary sequence with values in the unit interval being a diagonal of a projection.
Arveson was able to extend this index obstruction to any normal operator with finite spectrum coincident with its essential spectrum whose elements are the vertices of a convex polygon.
In order to state his main theorem, Arveson associated several objects to a finite set $X \subseteq \mathbb{C}$.

\begin{definition}
  \label{def:lim-x}
  For a finite set $X \subseteq \mathbb{C}$, the sequences which accumulate summably at $X$ are
  \begin{equation*}
    \Lim (X) := \left\{ (d_n) \in \ell^{\infty} \vmid \sum_{n=1}^{\infty} \dist(d_n,X) < \infty \right\}.
  \end{equation*}
\end{definition}

\begin{definition}
  \label{def:k_x}
  For a set $X = \{\lambda_1,\ldots,\lambda_m\} \subseteq \mathbb{C}$, let $K_X$ denote the $\mathbb{Z}$-module of linear combinations over $\mathbb{Z}$ of elements of $X$ whose coefficients sum to zero.
  This can also be expressed as the free $\mathbb{Z}$-module generated by $\lambda_1-\lambda_2, \ldots, \lambda_1-\lambda_m$.
\end{definition}

\begin{definition}
  \label{def:s-quotient}
  For a finite set $X \subseteq \mathbb{C}$ there is a natural map $s : \Lim(X) \to \mathbb{C}/K_X$.
  For $(d_n) \in \Lim(X)$, since $X$ is finite there are $x_n \in X$ for which $\abs{d_n - x_n} = \dist(d_n,X)$, and therefore the series $\sum_{n=1}^{\infty} (d_n - x_n)$ is absolutely summable.
  Arveson proved in \cite[Proposition~1]{Arv-2007-PNASU} that the coset of this sum in $\mathbb{C}/K_X$ is independent of the choices of $x_n \in X$, so the map
  \begin{equation*}
    s(d) := \sum_{n=1}^{\infty} (d_n - x_n) + K_X \in \mathbb{C}/K_X
  \end{equation*}
  is well-defined.
  The element $s(d)$ is called the \emph{renormalized sum} of $d$.
\end{definition}

We reproduce Arveson's theorem \cite[Theorem~4]{Arv-2007-PNASU} verbatim for reference.
Here, $\mathcal{N}(X)$ denotes the set of normal operators with finite spectrum $X$ coincident with their essential spectrum.

\begin{theorem}[Arveson]
  \label{thm:arveson-pythagorean}
  Let $X = \{\lambda_1,\ldots,\lambda_m\}$ be the set of vertices of a convex polygon $P \subseteq \mathbb{C}$ and let $d = (d_1,d_2,\ldots)$ be a sequence of complex numbers satisfying $d_n \in P$ for $n \ge 1$, together with the summability condition
  \begin{equation}
    \label{eq:summability-condition}
    \sum_{n=1}^{\infty} \abs{f(d_n)} < \infty,
  \end{equation}
  where $f(z) = (z-\lambda_1)(z-\lambda_2)\cdots(z-\lambda_m)$.
  Then $d \in \Lim(X)$; and if $d$ is the diagonal of an operator $N \in \mathcal{N}(X)$, then $s(d) = 0$.
\end{theorem}

The summability condition \eqref{eq:summability-condition} is equivalent to $(d_n) \in \Lim(X)$ via a routine analysis argument (see \cite[Proposition~2]{Arv-2007-PNASU}).
Moreover, using the notation of \autoref{def:s-quotient}, $s(d) = 0$ is equivalent by definition to the existence integers $c_1,\ldots,c_m$ (which depend on the choices $x_n \in X$) whose sum is zero for which
\begin{equation}
  \label{eq:c_k-integers}
  \sum_{n=1}^{\infty} (d_n - x_n) = \sum_{k=1}^m c_k \lambda_k \in K_X.
\end{equation}

When $X = \{0,1\} = \spec(N)$, $N$ is a projection, and the condition $(d_n) \in \Lim(\{0,1\})$ is equivalent to $a+b < \infty$, where $a,b$ are defined as in \autoref{thm:kadison-carpenter-pythagorean}.
Moreover, $K_{\{0,1\}} = \mathbb{Z}$, so that Arveson's theorem is a generalization of the forward implication \autoref{thm:kadison-carpenter-pythagorean}\ref{item:a+b-finite} in the situation where $P$ is an infinite and co-infinite projection.

Our focus is to provide a new approach to Arveson's theorem that, by linking it to the notion of diagonalization by unitaries which are Hilbert--Schmidt perturbations of the identity, permits us both to identify the integers $c_k$ of \eqref{eq:c_k-integers} implicit in the theorem in terms of essential codimension and to eliminate some of the hypotheses in the theorem.
Our intent is to bring a fresh perspective on two key parts of Arveson's theorem: the quantity $\sum_{n=1}^{\infty} (d_n-x_n)$ and the condition $(d_n) \in \Lim \big(\spec(N)\big)$.
We identify the sum $\sum_{n=1}^{\infty} (d_n-x_n)$ as $\trace\big(E(N-N')\big)$ for some diagonal operator $N'$ with $\spec(N') \subseteq \spec(N)$ (\autoref{prop:renormalized-sum-trace}).
Here $E : B(\mathcal{H}) \to \mathcal{A}$ denotes the canonical trace-preserving conditional expectation onto the atomic masa associated to an orthonormal basis;
that is, $E$ is the operation of ``taking the main diagonal.''
Then we prove that if $N$ is normal and $U$ is a unitary which is a Hilbert--Schmidt perturbation of the identity, then $E(N-UNU^{*})$ is trace-class and $\trace\big(E(N-UNU^{*})\big) = 0$ (\autoref{thm:expectation-trace-zero}).
Next, we establish that the condition $(d_n) \in \Lim \big(\spec(N)\big)$ is equivalent to the diagonalizability of $N$ by a unitary which is a Hilbert--Schmidt perturbation of the identity (\autoref{thm:diagonalizable-by-I-plus-HS}).
The proof relies on essential codimension and a geometric lemma (\autoref{lem:convexity-coefficient-corner}) which is similar to \cite[Lemma~1]{Arv-2007-PNASU}.
This culminates in a generalization of Arveson's theorem (\autoref{thm:arveson-reformulated}) proved using techniques involving essential codimension, which allows for the identification of the integers $c_k$ in terms of the essential codimensions of pairs of spectral projections of $N$ and a diagonal operator $N'$.
Finally, we show how our results may be used to derive Arveson's \autoref{thm:arveson-pythagorean}.

\section{Essential codimension}
\label{sec:essential-codimension}

A fundamental tool we use throughout is the notion of essential codimension due to Brown, Douglas and Fillmore \cite[Remark~4.9]{BDF-1973-PoaCoOT}.
It associates an integer to a pair of projections $P,Q$ whose difference is compact by means of the Fredholm operator $QP: P\Hil \to Q\Hil$.

\begin{definition}
  \label{def:essential-codimension}
  Given a pair of projections $P,Q$ whose difference is compact, the \emph{essential codimension} of $P$ in $Q$, denoted $[P:Q]$, is the integer defined by
  \begin{equation*}
    [P:Q] :=
    \begin{cases}
      \trace P-\trace Q & \text{if}\ \trace P,\trace Q < \infty, \\[0.5em]
      \ind(V^{*}W) & \parbox[c][2em]{0.5\textwidth}{if $\trace(P) = \trace(Q) = \infty$, where \\
       $W^{*}W = V^{*}V = I, WW^{*} = P, VV^{*} = Q$.} \\[0.4em]
    \end{cases}
  \end{equation*}
  Equivalently, essential codimension maybe be defined as
  \begin{equation*}
    [P:Q] := \ind(QP), \quad\text{where}\ QP : P\Hil \to Q\Hil.
  \end{equation*}
\end{definition}

Several simple properties of essential codimension which we use are collated here for reference.
Proofs can be found in, for example, \cite[Proposition~2.2]{BL-2012-CJM}.
Each property can be derived from standard facts about Fredholm index.

\begin{proposition}
  \label{prop:essential-codimension}
  Let $P_1,P_2$ and $Q_1,Q_2$ each be mutually orthogonal pairs of projections with the property that $P_j-Q_j$ is compact for $j=1,2$.
  Suppose also that $R_1$ is a projection for which $Q_1-R_1$ is compact.
  Then
  \begin{enumerate}
  \item\label{item:negative-reverse} $[P_1:Q_1] = -[Q_1:P_1]$
  \item\label{item:orthogonal-sum} $[P_1:Q_1] + [P_2:Q_2] = [P_1+P_2:Q_1+Q_2]$
  \item\label{item:concatenation} $[P_1:R_1] = [P_1:Q_1] + [Q_1:R_1]$
  \end{enumerate}
\end{proposition}

The original result of Brown, Douglas and Fillmore \cite[Remark~4.9]{BDF-1973-PoaCoOT} characterizes when projections can be conjugated by a unitary which is a compact perturbation of the identity.
More specifically, they proved that there is a unitary $U = I+K$ with $K$ compact which conjugates $P,Q$ if and only if $P-Q$ is compact and their essential codimension is zero.
The next proposition comes from \cite[Proposition~2.7(ii)]{KL-2017-IEOT} and extends the Brown--Douglas--Fillmore result verbatim to an arbitrary proper operator ideal $\mathcal{J}$, where $\mathcal{J}$ is two-sided but not necessarily norm-closed.
Herein, $\mathcal{J}$ will always denote a proper operator ideal.

\begin{proposition}
  \label{prop:restricted-conjugation-of-projections}
  If $P,Q$ are projections and $\mathcal{J}$ is a proper operator ideal, then $Q = UPU^{*}$ for some unitary $U = I+K$ with $K \in \mathcal{J}$ if and only if $P-Q \in \mathcal{J}$ and $[P:Q] = 0$.
\end{proposition}

The following proposition is a reformulation of \cite[Proposition~2.8]{KL-2017-IEOT} for the case when the ideal is the Hilbert--Schmidt class $\mathcal{C}_2$.
This proposition relates the Kadison integer to essential codimension in the following manner.
If $P$ is a projection with diagonal $(d_n)$ and $a,b$ are as in \autoref{thm:kadison-carpenter-pythagorean} with $a+b < \infty$, then, by choosing $Q$ to be the projection onto $\spans \{ e_n \mid d_n \ge \nicefrac{1}{2} \}$, \autoref{prop:kadison-integer-essential-codimension} guarantees $P-Q$ is Hilbert--Schmidt (a fact which was known to Arveson) and that $a-b = [P:Q]$.

\begin{proposition}
  \label{prop:kadison-integer-essential-codimension}
  Suppose $P,Q$ are projections. Then $P-Q$ is Hilbert--Schmidt if and only if in some (equivalently, every) orthonormal basis $\{e_n\}_{n=1}^{\infty}$ which diagonalizes $Q$, the diagonal $(d_n)$ of $P$ satisfies $a+b < \infty$, where
  \begin{equation*}
    a := \sum_{e_n \in Q^{\perp}\Hil} d_n = \trace(Q^{\perp}PQ^{\perp}) \quad\text{and}\quad b := \sum_{e_n \in Q\Hil} (1-d_n) = \trace(Q-QPQ).
  \end{equation*}
  Whenever $P-Q$ is Hilbert--Schmidt, $a-b = [P:Q]$.
\end{proposition}

\section{Restricted diagonalization}
\label{sec:restricted-diagonalization}

It is elementary that finite spectrum normal operators are diagonalizable.
However, one may ask about the possibility of diagonalization relative to a fixed orthonormal basis (or atomic masa) by a unitary of the form $U = I+K$ where $K$ lies in a given proper operator ideal $\mathcal{J}$.
For this we use the term \emph{restricted diagonalization}.
This concept has been studied by others in the aforementioned paper of Brown--Douglas--Fillmore \cite{BDF-1973-PoaCoOT}, as well as by Belti\c{t}a--Patnaik--Weiss \parencite{BPW-2016-IUMJ}, and Hinkkanen \cite{Hin-1985-MMJ}.
To our knowledge, the term restricted diagonalization was introduced by Belti\c{t}a--Patnaik--Weiss.

\subsection{Conditions for restricted diagonalization}

The next result is a corollary of \Cref{prop:restricted-conjugation-of-projections,prop:kadison-integer-essential-codimension}. It describes the conditions under which a projection experiences restricted diagonalization.
In the special case of the Hilbert--Schmidt ideal, this corollary shows that it suffices to examine the diagonal of the projection.

\begin{corollary}
  \label{cor:restricted-diagonalization-of-projections}
  For a projection $P$ and a proper operator ideal $\mathcal{J}$, the following are equivalent:
  \begin{enumerate}
  \item\label{item:diagonalization} $P$ is diagonalizable by a unitary $U=I+K$ with $K \in \mathcal{J}$;
  \item\label{item:diagonal-projection} there exists a diagonal projection $Q$ for which $P-Q \in \mathcal{J}$.
  \end{enumerate}
  If $\mathcal{J} = \mathcal{C}_2$, then these are also equivalent to:
  \begin{enumerate}[resume]
  \item\label{item:diagonal-sequence} the diagonal $(d_n)$ of $P$ lies in $\Lim(\{0,1\})$.
  \end{enumerate}
\end{corollary}

\begin{proof}
  \ref{item:diagonalization} $\Rightarrow$ \ref{item:diagonal-projection}.
  Suppose that $P$ is diagonalizable by a unitary $U=I+K$ with $K \in \mathcal{J}$.
  Then setting $Q := UPU^{*}$, we have $P-Q = -KP-PK^{*}-KPK^{*} \in \mathcal{J}$.

  \ref{item:diagonal-projection} $\Rightarrow$ \ref{item:diagonalization}.
  Suppose $Q$ is a diagonal projection for which $P-Q \in \mathcal{J}$.
  By replacing $Q$ with a diagonal projection $Q'$ that is a finite perturbation of $Q$, we can assume that $[P:Q] = 0$.
  Indeed, notice that if $[P:Q] < 0$, then $\trace Q \ge -[P:Q]$, so there is a diagonal subprojection $Q'$ of $Q$ with $\trace(Q-Q') = -[P:Q]$.
  Similarly, if $[P:Q] > 0$, then $\trace Q^{\perp} \ge [P:Q]$, so there is a diagonal subprojection $R$ of $Q^{\perp}$ with $\trace R = [P:Q]$, and in this case we set $Q' = Q+R$.
  In either case, the construction guarantees $[P:Q] = -[Q:Q']$, and hence by \autoref{prop:essential-codimension}\ref{item:concatenation}, $[P:Q'] = [P:Q] + [Q:Q'] = 0$.
  Therefore by \autoref{prop:restricted-conjugation-of-projections}, $P$ and $Q'$ are conjugated by a unitary $U=I+K$ with $K \in \mathcal{J}$, and hence $P$ is diagonalized by $U$.

  \ref{item:diagonal-projection} $\Rightarrow$ \ref{item:diagonal-sequence}.
  If $P-Q \in \mathcal{C}_2$, then by \autoref{prop:kadison-integer-essential-codimension}, for $a,b$ defined as in that proposition, $a+b < \infty$.
  Equivalently, $(d_n) \in \Lim(\{0,1\})$.

  \ref{item:diagonal-sequence} $\Rightarrow$ \ref{item:diagonal-projection}.
  If the diagonal $(d_n)$ of $P$ lies in $\Lim(\{0,1\})$, then there are some choices $x_n \in \{0,1\}$ for which $(d_n - x_n) \in \ell^1$.
  Let $Q$ be the diagonal projection onto the $\spans\{ e_n \mid x_n =1 \}$.
  Then for $a,b$ as defined in \autoref{prop:kadison-integer-essential-codimension}, $a+b < \infty$, and so $P-Q \in \mathcal{C}_2$ by that result.
\end{proof}

We will generalize \autoref{cor:restricted-diagonalization-of-projections} to finite spectrum normal operators.
The equivalence \ref{item:diagonalization}~$\Leftrightarrow$~\ref{item:diagonal-projection} is generalized by \autoref{thm:finite-spectrum-normal-restricted-diagonalizability}, and \ref{item:diagonalization}~$\Leftrightarrow$~\ref{item:diagonal-sequence} by \autoref{thm:diagonalizable-by-I-plus-HS}.

\autoref{prop:restricted-conjugation-of-projections} can be bootstrapped by induction to characterize when a pair of finite collections of mutually orthogonal projections can be \emph{simultaneously} conjugated by a unitary $U = I + K$ with $K \in \mathcal{J}$.

\begin{lemma}
  \label{lem:restricted-conjugation-of-sets-of-projections}
  Suppose $\{P_k\}_{k=1}^m, \{Q_k\}_{k=1}^m$ are each finite sets of mutually orthogonal projections, and $\mathcal{J}$ is a proper operator ideal.
  Then there is some unitary $U = I+K$ with $K \in \mathcal{J}$ for which $Q_k = UP_kU^{*}$ for $1 \le k \le m$ if and only if $P_k - Q_k \in \mathcal{J}$ and $[P_k:Q_k] = 0$ for all $1 \le k \le m$.
\end{lemma}

\begin{proof}
  One direction is straightforward.
  Namely, if there exists a unitary $U = I+K$ with $K \in \mathcal{J}$ for which $Q_k = UP_k U^{*}$ for all $1 \le k \le m$, then by \autoref{prop:restricted-conjugation-of-projections} $P_k - Q_k \in \mathcal{J}$ and $[P_k:Q_k] = 0$.

  For the other direction, we use induction on $m$, and the base case $m=1$ follows from \autoref{prop:restricted-conjugation-of-projections}.
  Let $m \in \mathbb{N}$ and suppose that if $\{P_k\}_{k=1}^m, \{Q_k\}_{k=1}^m$ are each sets of mutually orthogonal projections and satisfy $P_k - Q_k \in \mathcal{J}$ and $[P_k:Q_k] = 0$, then there is a single unitary $U=I+K$ with $K \in \mathcal{J}$ which conjugates $P_k$ into $Q_k$, i.e., $Q_k = UP_kU^{*}$.

  Now suppose we have two sets of $m+1$ mutually orthogonal projections satisfying these conditions.
  By \autoref{prop:restricted-conjugation-of-projections} there is a unitary $V = I+K$ with $K \in \mathcal{J}$ for which $Q_{m+1} = VP_{m+1}V^{*}$.
  Moreover, for $1 \le k \le m$, $P'_k := VP_k V^{*}$ satisfies $P_k - P'_k \in \mathcal{J}$ and $[P_k:P'_k] = 0$.
  Therefore $P'_k - Q_k \in \mathcal{J}$ and $[P'_k:Q_k] = 0$ by \autoref{prop:essential-codimension}\ref{item:concatenation}.
  Applying the inductive hypothesis to the collections $\{P'_k\}_{k=1}^m, \{Q_k\}_{k=1}^m$ on the Hilbert space $Q_{m+1}^{\perp} \Hil$ yields a unitary $W = Q_{m+1}^{\perp} + K'$ acting on $Q_{m+1}^{\perp} \Hil$ with $K' \in \mathcal{J}$, and which conjugates $P'_k$ into $Q_k$ for $1 \le k \le m$.
  Extending this to the unitary $Q_{m+1} \oplus W$ acting on $\Hil$ and setting $U = (Q_{m+1} \oplus W)V$, we find that $U$ is of the desired form and $UP_k U^{*} = Q_k$ for $1 \le k \le m+1$.
\end{proof}

The following lemma weakens the sufficient condition of \autoref{lem:restricted-conjugation-of-sets-of-projections} so long as we are allowed to perturb the diagonal projections.

\begin{lemma}
  \label{lem:sum-zero-each-zero}
  Suppose that $\{P_k\}_{k=1}^m, \{Q_k\}_{k=1}^m$ are each collections of mutually orthogonal projections for which $P_k - Q_k \in \mathcal{J}$ and $\sum_{k=1}^m [P_k:Q_k] = 0$.
  Then for every atomic masa $\mathcal{A}$ containing $\{Q_k\}_{k=1}^m$, there exist mutually orthogonal projections $\{Q'_k\}_{k=1}^m \subseteq \mathcal{A}$ for which $P_k - Q'_k \in \mathcal{J}$ and $[P_k:Q'_k] = 0$.
\end{lemma}

\begin{proof}
  Suppose $\{Q_k\}_{k=1}^m$ lies in an atomic masa.
  Note that such a masa always exists since this is a finite collection of mutually orthogonal (hence commuting) projections.
  The argument is by induction on $m$.
  When $m=1$, the claim is trivial.

  Now suppose $m > 1$.
  Either $[P_k : Q_k] = 0$ for all $k$ already, or there are two indices $1 \le i,j \le m$ with $[P_i : Q_i] < 0 < [P_j:Q_j]$.
  Notice that $\trace Q_i \ge -[P_i : Q_i]$.
  Let $Q$ be a diagonal subprojection of $Q_i$ with $\trace Q = \min \{ -[P_i:Q_i], [P_j:Q_j] \}$.
  Then we replace $Q_i$ with $Q_i - Q$ and $Q_j$ with $Q_j + Q$.
  By construction, either $\big[P_i:(Q_i - Q)\big] = 0$ or $\big[P_j:(Q_j+Q)\big] = 0$.
  So now we have $n-1$ pairs of projections for which the sum of the essential codimensions is zero.
  By induction we can actually force them all to be zero while maintaining the condition that the $Q'_k$ projections are diagonal.
\end{proof}

\begin{theorem}
  \label{thm:finite-spectrum-normal-restricted-diagonalizability}
  Suppose $\mathcal{J}$ is a proper operator ideal.
  A finite spectrum normal operator is diagonalizable by a unitary $U=I+K$ with $K \in \mathcal{J}$ if and only if each spectral projection differs from a diagonal projection by an element of $\mathcal{J}$.
\end{theorem}

\begin{proof}
  Let $N = \sum_{k=1}^m \lambda_k P_k$ be a finite spectrum normal operator with spectral projections $P_k$ associated to the eigenvalues $\lambda_k$.
  One direction is trivial, namely, if $N$ is diagonalizable by a unitary $U=I+K$ with $K \in \mathcal{J}$, then the projections $Q_k := UP_kU^{*}$ are diagonal and $P_k - Q_k \in \mathcal{J}$.
  
  For the other direction, suppose that for each $P_k$ there is a diagonal projection $Q_k$ for which $P_k - Q_k \in \mathcal{J}$.
  The operators $Q_jQ_k$ are projections because $Q_j,Q_k$ are commuting projections. Then since $P_jP_k = \delta_{jk}P_j$, for $j\not=k$ we obtain
  \begin{equation}
    \label{eq:Q_j*Q_k-in-J}
    \begin{split}
      Q_jQ_k &= \big(P_j+(Q_j-P_j)\big)\big(P_k+(Q_k-P_k)\big) \\
      &= (Q_j-P_j)P_k + P_j(Q_k-P_k) + (Q_j-P_j)(Q_k-P_k) \in \mathcal{J}.
    \end{split}
  \end{equation}
  Therefore $Q_jQ_k$ are finite projections when $j \not= k$.

  Now let $Q'_1 := Q_1$ and inductively define $Q'_j = Q_j-Q_j(Q'_1+\cdots+Q'_{j-1})$ for $1 < j < m$ and finally $Q'_m = I-(Q'_1+\cdots+Q'_{m-1})$.
  It is clear that for $1 \le j < m$, $Q'_j$ is in the algebra generated by $\{Q_1,\ldots,Q_j\}$ and is therefore diagonal.
  Moreover, for $1 \le j < m$, by \eqref{eq:Q_j*Q_k-in-J} and induction $Q'_j - Q_j$ is finite rank, and hence $P_j - Q'_j \in \mathcal{J}$.
  Thus, $Q'_m$ is a $\mathcal{J}$-perturbation of $I-(P_1+\cdots+P_{m-1}) = P_m$, and hence $P_m - Q'_m \in \mathcal{J}$ as well.
  By \autoref{prop:essential-codimension}(ii),
  \begin{equation*}
    \sum_{k=1}^m [P_k:Q'_k] = \left[ \sum_{k=1}^m P_k: \sum_{k=1}^m Q'_k \right] = [I:I] = 0.
  \end{equation*}
  So, by \autoref{lem:sum-zero-each-zero}, we may assume by passing to a possibly different collection of diagonal $Q'_k$ that, in fact, $[P_k:Q'_k] = 0$ for $1 \le k \le m$.
  Finally, by \autoref{lem:restricted-conjugation-of-sets-of-projections} there is a unitary $U=I+K$ with $K \in \mathcal{J}$ for which $Q'_k = UP_k U^{*}$ for each $1 \le k \le m$.
  Therefore, $UNU^{*} = \sum_{k=1}^m \lambda_k Q'_k$, which is a diagonal operator.
\end{proof}

\subsection{Consequences of restricted diagonalization}

This subsection is motivated by the following observation about the condition $(d_n) \in \Lim\big(\spec(N)\big)$ in Arveson's theorem.

\begin{proposition}
  \label{prop:renormalized-sum-trace}
  Let $N$ be a normal operator with finite spectrum and let $(d_n)$ be the diagonal of $N$.
  Then $(d_n) \in \Lim\big(\spec(N)\big)$ if and only if there exists a diagonal operator $N' = \diag(x_n)$ such that $spec(N') \subseteq \spec(N)$, and $E(N-N')$ is trace-class, in which case
  \begin{equation}
    \label{eq:renormalized-sum-trace}
    \trace\big(E(N-N')\big) = \sum_{n=1}^{\infty} (d_n - x_n).
  \end{equation}
\end{proposition}

\begin{proof}
  ($\Rightarrow$)
  Suppose $(d_n) \in \Lim\big(\spec(N)\big)$.
  Then there is a sequence $(x_n)$ with $x_n \in \spec(N)$ such that $(d_n - x_n)$ is absolutely summable, and we may take $N' := \diag(x_n)$.
  Therefore, since $(d_n - x_n)$ is absolutely summable,
  \begin{equation*}
    \trace \abs{E(N-N')} = \sum_{n=1}^{\infty} \abs{d_n-x_n} < \infty,
  \end{equation*}
  and hence $E(N-N')$ is trace-class.

  ($\Leftarrow$)
  Suppose $N'$ is a diagonal operator with $\spec(N') \subseteq \spec(N)$ and $E(N-N')$ trace-class, and let $(x_n)$ denote the diagonal of $N'$.
  Then $x_n \in \spec(N') \subseteq \spec(N)$ and since $E(N-N')$ is trace-class,
  \begin{equation*}
    \sum_{n=1}^{\infty} \abs{d_n-x_n} = \trace \abs{E(N-N')} < \infty.
  \end{equation*}
  Therefore $(d_n-x_n)$ is absolutely summable and hence $d_n \in \Lim\big(\spec(N)\big)$.

  Notice that whenever either of the equivalent conditions is satisfied, we have the equality
  \begin{equation*}
    \trace\big(E(N-N')\big) = \sum_{n=1}^{\infty} (d_n - x_n). \qedhere
  \end{equation*}
\end{proof}

The remainder of the section is devoted to analyzing the expression $E(N-N')$ when $N'$ is a restricted diagonalization of a normal operator $N$ (not necessarily with finite spectrum), i.e., when $N' = UNU^{*}$ where $U = I + K$ is unitary and $K \in \mathcal{J}$.

As in \cite{DFWW-2004-AM}, the \emph{arithmetic mean closure} $\mathcal{J}^{-}$ of an operator ideal $\mathcal{J}$ is the set of operators $T$ whose singular values are weakly majorized by the singular values of an operator $A \in \mathcal{J}$; that is, if $s(T)$ denotes the singular value sequence of a compact operator $T$,
\begin{equation*}
  \mathcal{J}^- := \left\{ T \in B(\Hil) \vmid \exists B \in \mathcal{J}, \forall n\in \mathbb{N},\  \sum_{j=1}^n s_j(T) \le \sum_{j=1}^n s_j(B) \right\}.
\end{equation*}
An ideal $\mathcal{J}$ is said to be \emph{arithmetic mean closed} if $\mathcal{J} = \mathcal{J}^-$.
Common examples of arithmetic mean closed ideals are the Schatten ideals $\mathcal{C}_p$ of which the trace-class ideal $\mathcal{C}_1$ and Hilbert--Schmidt ideal $\mathcal{C}_2$ are special cases.

In \cite{KW-2011-IUMJ}, Kaftal and Weiss investigated the relationship between an ideal $\mathcal{J}$ and the elements of its image $E(\mathcal{J})$ under a trace-preserving conditional expectation onto an atomic masa $\mathcal{A}$, and they established the following characterization \cite[Corollary~4.4]{KW-2011-IUMJ}.

\begin{corollary}
  \label{cor:diagonal-invariance}
  For every operator ideal $\mathcal{J}$, $E(\mathcal{J}) = \mathcal{J}^- \cap \mathcal{A}$.
\end{corollary}

Our next result says if an operator $N$ can be diagonalized by a unitary $U=I+K$ with $K \in \mathcal{J}$ then the diagonals of $N$ and its diagonalization differ by an element of the arithmetic mean closure of $\mathcal{J}^2$.

\begin{proposition}
  \label{prop:diagonalization-by-I-plus-K-necessity}
  Let $N$ be a diagonal operator, $\mathcal{J}$ an operator ideal, and $U = I+K$ a unitary with $K \in \mathcal{J}$.
  Then $E(UNU^{*}-N) \in (\mathcal{J}^2)^-$.
\end{proposition}

\begin{proof}
  Irrespective of the condition $K \in \mathcal{J}$, note that $U = I+K$ is unitary if and only if $K$ is normal and $K+K^{*} = -K^{*}K$ because
  \begin{align*}
    UU^{*} &= I + K + K^{*} + KK^{*} \\
    U^{*}U &= I + K + K^{*} + K^{*}K.
  \end{align*}

  Then
  \begin{align*}
    E(UNU^{*}-N) &= E(KN+NK^{*}+KNK^{*}) \\
                 &= E(KN)+E(NK^{*}) + E(KNK^{*}) \\
                 &= E(K)N+NE(K^{*}) + E(KNK^{*}) \\
                 &= E(K+K^{*})N + E(KNK^{*})  \in (\mathcal{J}^2)^-,
  \end{align*}
  by \autoref{cor:diagonal-invariance}.
\end{proof}

When $\mathcal{J} = \mathcal{C}_2$, which is the primary concern in this paper, we can say more.

\begin{theorem}
  \label{thm:expectation-trace-zero}
  Suppose $N$ is a normal operator.
  There is an atomic masa such that for every unitary $U = I + K$ with $K$ Hilbert--Schmidt, $E(UNU^{*}-N)$ is trace-class and has trace zero.
  Moreover, if $N$ is diagonalizable, any atomic masa containing $N$ suffices.
\end{theorem}

\begin{proof}
  Suppose first that $N$ is diagonalizable and consider an atomic masa in which $N$ lies.
  Let $U = I+K$ be unitary with $K$ Hilbert--Schmidt.
  By \autoref{prop:diagonalization-by-I-plus-K-necessity} with $\mathcal{J} = \mathcal{C}_2$ and its proof, each term of $E(UNU^{*}-N) = E(K+K^{*})N + E(KNK^{*})$ is trace-class because $K+K^{*} = -K^{*}K$ and $KNK^{*}$ are trace-class, and because the trace-class is arithmetic mean closed (in fact, it is the \emph{smallest} arithmetic mean closed ideal).
  Then, because the conditional expectation is trace-preserving, we find
  \begin{align*}
    \trace\big(E(KNK^{*})\big) &= \trace(KNK^{*}) = \trace(K^{*}KN) \\
                       &= -\trace((K+K^{*})N) = -\trace(E(K+K^{*})N),
  \end{align*}
  and therefore $\trace\big(E(UNU^{*}-N)\big) = 0$.

  Now suppose $N$ is an arbitrary normal operator.
  By Voiculescu's extension \cite{Voi-1979-JOT} of the Weyl--von Neumann--Berg theorem we can write $N = D+J$ where $D$ is diagonalizable and $J$ is Hilbert--Schmidt.
  Then $UJU^{*}-J = KJ+JK^{*}+KJK^{*}$ and each term is trace-class.
  Moreover,
  \begin{align*}
    \trace(KJK^{*}) &= \trace(K^{*}KJ) = -\trace((K+K^{*})J) \\
                    &= -\trace(KJ)-\trace(K^{*}J) = -\trace(KJ)-\trace(JK^{*}),
  \end{align*}
  and hence $\trace(UJU^{*}-J) = 0$.
  Therefore, if $E$ is a conditional expectation onto an atomic masa containing $D$, then $E(UNU^{*}-N) = E(UDU^{*}-D) + E(UJU^{*}-J)$ has trace zero.
\end{proof}

The previous theorem establishes a kind of \emph{trace invariance} property for arbitrary normal operators.
To see why we use this terminology, consider that a trace-class operator $A$ has a trace which is invariant under unitary conjugation.
That is, for any unitary $U$, $\trace A = \trace (UAU^{*})$.
Rearranging, we can write this as $\trace(UAU^{*}-A) = 0$, and since the canonical expectation is trace-invariant, we can rewrite this as $\trace\big(E(UAU^{*}-A)\big) = 0$.
Under more restrictive hypotheses, \autoref{thm:expectation-trace-zero} ensures the same condition for normal operators instead of trace-class operators.

\begin{remark}
  The reader may have noticed that the normality in the previous theorem was only used in order to write the operator as a Hilbert--Schmidt perturbation of a diagonal operator.
  Therefore, the above theorem remains valid under this substitution of the hypothesis, and a slightly more general result is obtained.
\end{remark}

\begin{example}
  One may wonder if in \autoref{prop:diagonalization-by-I-plus-K-necessity} and \autoref{thm:expectation-trace-zero} we may take any trace-preserving conditional expectation instead of the special ones chosen.
  The answer is negative in general as this example shows.
  Consider commuting positive operators $C,S$ in $B(\Hil)$ with zero kernel satisfying $C^2 + S^2 = I$.
  Then consider the operators $P,U \in M_2\big(B(\Hil)\big) \cong B(\Hil \oplus \Hil)$
  \begin{equation*}
    P :=
    \frac{1}{\sqrt{2}}
    \begin{pmatrix}
      I & I \\
      I & I \\
    \end{pmatrix}
    \qquad
    U :=
    \begin{pmatrix}
      C & S \\
      -S & C \\
    \end{pmatrix},
  \end{equation*}
  which are a projection and a unitary, respectively.
  Thus
  \begin{equation*}
    UPU^{*} =
    \frac{1}{\sqrt{2}}
    \begin{pmatrix}
      I+2CS & C^2 - S^2 \\
      C^2 - S^2 & I-2CS \\
    \end{pmatrix}
  \end{equation*}
  Now, choose $S = \diag(\sin(\theta_n))$ and $C = \diag(\cos(\theta_n))$ with $(\theta_n) \in \ell^2 \setminus \ell^1$.
  Then $S \in \mathcal{C}_2, C-I \in \mathcal{C}_1$ and hence $U-(I \oplus I) \in \mathcal{C}_2$.
  Moreover, $2CS = \diag(\sin(2\theta_n))$ which is Hilbert--Schmidt but not trace-class.
  Thus, if $E$ is the expectation onto an atomic masa containing $C,S$, and $\tilde{E} := E \oplus E$, then
  $\tilde{E}(UPU^{*}-P) = \frac{1}{\sqrt{2}} (2CS \oplus -2CS) \in \mathcal{C}_2 \setminus \mathcal{C}_1$.
\end{example}

\section{Arveson's Theorem Revisited}
\label{sec:arveson}

In this section we apply the results concerning restricted diagonalization to prove a few key facts which will yield a reformulation and extension of Arveson's theorem (\autoref{thm:arveson-reformulated}).
Our first result in this direction is \autoref{thm:diagonalizable-by-I-plus-HS} which characterizes the condition $(d_n) \in \Lim\big(\spec(N)\big)$ in terms of restricted diagonalization.
In order to prove \autoref{thm:diagonalizable-by-I-plus-HS}, we use a straightforward geometric lemma which serves a similar purpose as \cite[Lemma~1]{Arv-2007-PNASU}.

\begin{lemma}
  \label{lem:convexity-coefficient-corner}
  Suppose $\lambda_1,\ldots,\lambda_m \in \mathbb{C}$ are distinct and $x = \sum_{j=1}^m c_j \lambda_j$ is a convex combination, and $L$ is a line separating $\lambda_k$ from the remaining $\lambda_j$.
  If $x$ lies on a line parallel to $L$ separating $\lambda_k$ from $L$, then
  \begin{equation*}
    \sum_{\substack{j=1 \\ j\not=k}}^m c_j \le \frac{\abs{x - \lambda_k}}{\dist(\lambda_k, L)}.
  \end{equation*}
\end{lemma}

\begin{proof}
  Relabel the $\lambda_j$ if necessary so that $k=1$.
  By applying a rotation, translation and scaling (which preserve proportional distances), we may suppose that $\lambda_1 = 1$ and $L = -a + i\mathbb{R}$ for some $a \ge 0$ so that the real part $\Re(x) = 0$.
  Note that $-a \ge \max_{j \ge 2} \{ \Re(\lambda_j)\}$.
  Since $0 \in [-a,1]$ we may write
  \begin{equation*}
    t (-a) + (1-t) 1 = 0, \quad\text{for}\quad t = \frac{1}{1 + a}
  \end{equation*}
  Now
  \begin{equation*}
    0 = \Re(x) = \sum_{j=1}^m c_j \Re(\lambda_j) \le \Bigg( \sum_{j=2}^m c_j \Bigg) \max_{j \ge 2} \{\Re(\lambda_j)\} + c_1 \lambda_1 \le \Bigg( \sum_{j=2}^m c_j \Bigg) (-a) + c_1 1.
  \end{equation*}
  Since we have two convex combinations of $-a, 1$ and the latter is closer to $1$ than the former, the convexity coefficients satisfy
  \begin{equation*}
    \sum_{j = 2}^m c_j \le t = \frac{1}{1+a} = \frac{\dist(\Re(x),\lambda_1)}{\dist(\lambda_1,L)} \le  \frac{\abs{x-\lambda_1}}{\dist(\lambda_1,L)}. \qedhere
  \end{equation*}
\end{proof}

\begin{theorem}
  \label{thm:diagonalizable-by-I-plus-HS}
  Let $N$ be a normal operator with finite spectrum and diagonal $(d_n)$, and let $X$ be the vertices of the convex hull of its essential spectrum.
  Then $(d_n) \in \Lim(X)$ if and only if $\essspec(N) = X$ and $N$ is diagonalizable by a unitary which is a Hilbert--Schmidt perturbation of the identity.
\end{theorem}

\begin{proof}
  We first reduce to the case when $\spec(N) = \essspec(N)$.
  Since $N$ is a normal operator with finite spectrum, by the spectral theorem there is a finite rank perturbation $N'$ of $N$ for which $N'$ is normal and $\spec(N') = \essspec(N') = \essspec(N)$.
  In particular, if $P_{\lambda}$ are the spectral projections of $N$ onto $\{\lambda\}$, and $\lambda' \in \essspec(N)$ is a distinguished element, then we can choose
  \begin{equation*}
    N' := \lambda' P + \sum_{\lambda \in \essspec(N)} \lambda P_{\lambda}, \quad\text{where}\quad P = \sum_{\lambda \notin \essspec(N)} P_{\lambda}.
  \end{equation*}
  Since $N'-N$ is finite rank, the diagonals of $N'$ and $N$ differ by an absolutely summable sequence, so $(d_n) \in \Lim(X)$ if and only if the diagonal of $N'$ is in $\Lim(X)$.
  Moreover, the spectral projections of $N$ and $N'$ differ from one another by finite projections.
  Therefore, the spectral projections of $N$ each differ from a diagonal projection by a Hilbert--Schmidt operator if and only if the same holds true for $N'$.
  By \autoref{thm:finite-spectrum-normal-restricted-diagonalizability}, $N$ is diagonalizable by a unitary which is a Hilbert--Schmidt perturbation of the identity if and only if $N'$ is as well.
  Therefore, by the above reduction, it suffices to prove the theorem with the added assumption that $\spec(N) = \essspec(N)$.
  
  (Proof of $\Rightarrow$)
  Enumerate the elements of $\essspec(N) = \spec(N)$ as $\lambda_1,\ldots,\lambda_m$.
  Let $P_j$ denote the spectral projection corresponding to the eigenvalue $\lambda_j$, so that $N = \sum_{j=1}^m \lambda_j P_j$.
  Let $\{e_n\}_{n=1}^{\infty}$ denote the orthonormal basis corresponding to the diagonal $(d_n)$.
  Suppose $(d_n) \in \Lim(X)$, and so there exist $x_n \in X$ for which $(d_n - x_n) \in \ell^1$.
  Let $\Lambda_k := \{ n \in \mathbb{N} \mid x_n = \lambda_k\}$ be the index set where the sequence $(x_n)$ takes the value $\lambda_k \in X$.

  The projections $P_j$ sum to the identity, so for each $n \in \mathbb{N}$, $\sum_{j=1}^m \angles{P_j e_n,e_n} = 1$ and therefore
  \begin{equation*}
    d_n = \sangles{Ne_n, e_n} = \sum_{j=1}^m \sangles{P_j e_n, e_n} \lambda_j
  \end{equation*}
  is a convex combination of the spectrum.

  For $\lambda_k \in X$, let $L_k$ be a line separating $\lambda_k$ from the remaining elements of $\essspec(N)$.
  Such a line $L_k$ exists because $\lambda_k$ is an extreme point of the convex hull of $\essspec(N)$, and this is a finite set.
  Since $(d_n) \in \Lim (X)$ we know that $(d_n - \lambda_k)_{n \in \Lambda_k}$ is absolutely summable for every $k$.
  Therefore, for all but finitely many indices $n \in \Lambda_k$, the diagonal entry $d_n$ lies on a line parallel to $L_k$ separating $\lambda_k$ from $L_k$ and hence also $\essspec(N) \setminus \{ \lambda_k \}$.

  By \autoref{lem:convexity-coefficient-corner}, for these indices $n \in \Lambda_k$,
  \begin{equation}
    \label{eq:P_j-in-ell^1}
    \sum_{\substack{j=1 \\ j\not=k}}^m \sangles{P_j e_n, e_n} \le \frac{\abs{d_n - \lambda_k}}{\dist(\lambda_k,L_k)}.
  \end{equation}
  Since this inequality holds for all but finitely many $n \in \Lambda_k$, and $\dist(\lambda_k,L_k)$ is independent of $n \in \Lambda_k$, and $(d_n - \lambda_k)_{n \in \Lambda_k}$ is absolutely summable, \eqref{eq:P_j-in-ell^1} proves $\big(\langle P_j e_n, e_n \rangle\big)_{n \in \Lambda_k}$ lies in $\Lim (\{0\}) = \ell^1$ when $j \not= k$.
  If $\lambda_j \in \essspec(N) \setminus X$, by letting $\lambda_k$ run through $X$, we find $\big(\langle P_j e_n, e_n \rangle\big)_{n \in \mathbb{N}}$ is absolutely summable since $\bigcup_{\lambda_k \in X} \Lambda_k = \mathbb{N}$.
  This implies $P_j$ is trace-class and hence a finite projection, contradicting the fact that $\lambda_j \in \essspec(N)$.
  Therefore $X = \essspec(N)$.

  Now consider $\lambda_j \in X = \essspec(N)$.
  In analogy with the previous paragraph, using the fact that $\big(\langle P_j e_n, e_n \rangle\big)_{n \in \Lambda_k} \in \ell^1$ when $j \not= k$ and letting $\lambda_k$ run through $X \setminus \lambda_j$, we find $\big(\langle P_j e_n, e_n \rangle\big)_{n \notin \Lambda_j} \in \ell^1$.
  Finally, for $n \in \Lambda_j$,
  \begin{equation*}
    1 - \sangles{P_j e_n, e_n} = \sum_{\substack{k=1 \\ k\not=j}}^m \sangles{P_k e_n, e_n},
  \end{equation*}
  and hence $\big(1- \sangles{P_j e_n, e_n}\big)_{n \in \Lambda_k}$ is a finite sum of absolutely summable sequences, and is therefore absolutely summable.
  Thus $\big(\sangles{P_j e_n, e_n}\big)_{n \in \Lambda_k} \in \Lim(\{1\})$, so $\big(\sangles{P_j e_n,e_n}\big) \in \Lim(\{0,1\})$.
  Therefore, by \autoref{cor:restricted-diagonalization-of-projections}, $P_j$ differs from a diagonal projection by a Hilbert--Schmidt operator.
  Since this is true of all the spectral projections of $N$, we may apply \autoref{thm:finite-spectrum-normal-restricted-diagonalizability} to conclude that $N$ is a diagonalizable by a Hilbert--Schmidt perturbation of the identity.

  (Proof of $\Leftarrow$) This implication is a direct corollary of \autoref{thm:expectation-trace-zero}.
  To see this, suppose $\essspec(N) = X$ and $N$ is diagonalizable by a unitary $U$ which is a Hilbert--Schmidt perturbation of the identity.
  Thus $UNU^{*} = \diag(x_n)$ for some sequence $x_n \in \essspec(N) = X$.
  Then by \autoref{thm:expectation-trace-zero}, $E(N-UNU^{*})$ is trace-class.
  That is, $\trace \big(E(N-UNU^{*})\big) = \sum_{n=1}^{\infty} (d_n-x_n)$ is an absolutely summable series, so $(d_n) \in \Lim(X)$. 
\end{proof}

We now establish our generalized operator-theoretic reformulation of Arveson's \autoref{thm:arveson-pythagorean} by means of \autoref{thm:expectation-trace-zero}.
After the proof we will explain how to derive \autoref{thm:arveson-pythagorean} from \autoref{thm:arveson-reformulated}.

\begin{theorem}
  \label{thm:arveson-reformulated}
  Let $N$ be a normal operator with finite spectrum.
  If $N$ is diagonalizable by a unitary which is a Hilbert--Schmidt perturbation of the identity,
  then there is a diagonal operator $N'$ with $\spec(N') \subseteq \spec(N)$ for which $E(N-N')$ is trace-class, and for any such $N'$, $\trace\big(E(N-N')\big) \in K_{\spec(N)}$.
  In particular,
  \begin{equation}
    \label{eq:trace-in-K-spec-N}
    \trace\big(E(N-N')\big) = \sum_{\lambda \in \spec(N)} [P_{\lambda}:Q_{\lambda}] \lambda,
  \end{equation}
  where $P_{\lambda},Q_{\lambda}$ are the spectral projections onto $\{\lambda\}$ of $N,N'$ respectively.
  Moreover, $P_{\lambda}-Q_{\lambda}$ is Hilbert--Schmidt for each $\lambda \in \spec(N)$.
\end{theorem}

\begin{proof}
  Suppose $N$ is normal operator with finite spectrum which is diagonalizable by a unitary $U$ that is a Hilbert--Schmidt perturbation of the identity.
  Then by \autoref{thm:expectation-trace-zero}, $E(UNU^{*}-N)$ is trace-class with trace zero.
  Moreover, $\spec(UNU^{*}) = \spec(N)$, thereby proving that an $N'$ as in the statement exists.
  
  Now, let $N'$ be any diagonal operator with $\spec(N') \subseteq \spec(N)$ for which $E(N-N')$ is trace-class.
  Since $N'$ and $UNU^{*}$ are diagonal, we find
  \begin{equation}
    \label{eq:diagonal-split}
    UNU^{*}-N' = E(UNU^{*}-N') = E(UNU^{*}-N)+E(N-N')
  \end{equation}
  is trace-class, diagonal, and has finite spectrum contained in the set of differences $\spec(N)-\spec(N)$.
  Together, these conditions imply this operator is finite rank.
  Moreover, the (diagonal) spectral projections of $UNU^{*}, N'$, which we denote $R_{\lambda},Q_{\lambda}$, respectively for $\lambda \in \spec(N)$, each differ by a finite rank operator.
  Here we allow for the case $Q_{\lambda} = 0$ when $\lambda \in \spec(N) \setminus \spec(N')$.
  This guarantees
  \begin{equation*}
    [R_{\lambda}:Q_{\lambda}] = \trace(R_{\lambda}-Q_{\lambda}),
  \end{equation*}
  using, for example, \autoref{prop:kadison-integer-essential-codimension}; however, this formula for essential codimension holds whenever the difference of the projections is trace-class and is widely known (see for instance \cite[Theorem~4.1]{ASS-1994-JFA}, \cite[Theorem~3]{AS-1994-LAA}, or \cite[Corollary~3.3]{CP-2004-KT}).

  Therefore,
  \begin{equation}
    \label{eq:trace-ess-codim-formula}
    \trace(UNU^{*}-N') = \trace \left( \sum_{\lambda \in \spec(N)} (\lambda R_{\lambda} - \lambda Q_{\lambda}) \right) = \sum_{\lambda \in \spec(N)} [R_{\lambda} : Q_{\lambda}] \lambda.
  \end{equation}
  Moreover, we can replace $R_{\lambda}$ with $P_{\lambda}$ in the right-most side of the above display.
  Indeed, since $U$ conjugates $P_{\lambda}, R_{\lambda}$, $[P_{\lambda} : R_{\lambda}] = 0$ by \autoref{prop:restricted-conjugation-of-projections}, and furthermore $[P_{\lambda} : Q_{\lambda}] = [P_{\lambda} : R_{\lambda}] + [R_{\lambda} : Q_{\lambda}]$ by \autoref{prop:essential-codimension}(iii).

  Finally, since $\trace\big(E(UNU^{*}-N)\big) = 0$, using \eqref{eq:diagonal-split} and \eqref{eq:trace-ess-codim-formula} we find that
  \begin{equation*}
    \trace\big(E(N-N')\big) = \trace(UNU^{*}-N') = \sum_{\lambda \in \spec(N)} [P_{\lambda}:Q_{\lambda}] \lambda. \qedhere
  \end{equation*}
\end{proof}

We now illustrate how our results may be used to provide a new proof of Arveson's theorem.

\begin{proof}[Proof of \autoref{thm:arveson-pythagorean}]
  Let $X = \{\lambda_1,\ldots,\lambda_m\}$ and $d = (d_1,d_2,\ldots)$ be as in \autoref{thm:arveson-pythagorean}.
  That is, $X$ is the set of vertices of a convex polygon in $\mathbb{C}$, and $d$ satisfies
  \begin{equation*}
    \sum_{n=1}^{\infty} \abs{f(d_n)} < \infty,
  \end{equation*}
  where $f(z) = (z-\lambda_1)(z-\lambda_2)\cdots(z-\lambda_m)$.
  As we remarked after \autoref{thm:arveson-pythagorean}, this summability condition is equivalent to $d \in \Lim (X)$ by \cite[Proposition~2]{Arv-2007-PNASU}.
  Now suppose $d$ is the diagonal of an operator $N \in \mathcal{N}(X)$ (i.e., $N$ is normal with $\spec(N) = \essspec(N) = X$).
  Then by \autoref{thm:diagonalizable-by-I-plus-HS}, $N$ is diagonalizable by a unitary $U = I+K$ with $K$ Hilbert--Schmidt.
  Therefore, we may apply \autoref{thm:arveson-reformulated} to conclude that $\trace\big(E(N-N')\big) \in K_{\spec(N)} = K_X$ for some diagonal operator $N'$ with $\spec(N') \subseteq \spec(N)$ and $E(N-N')$ is trace-class.
  Finally, equation \eqref{eq:renormalized-sum-trace} of \autoref{prop:renormalized-sum-trace} establishes
  \begin{equation*}
    \sum_{n=1}^{\infty} (d_n - x_n) = \trace\big(E(N-N')\big) \in K_X
  \end{equation*}
  where $(x_n)$ is the diagonal of $N'$, so $x_n \in \spec(N') = \spec(N) = X$.
  Hence $s(d) = 0$.
\end{proof}

\begin{remark}
  \label{rem:generalize-bownik-jasper}
  In \cite{BJ-2015-TAMS}, Bownik and Jasper completely characterized the diagonals of selfadjoint operators with finite spectrum.
  A few of the results we have presented herein are generalizations of \cite[Theorem~4.1]{BJ-2015-TAMS}, which consists of some necessary conditions for a sequence to be the diagonal of a finite spectrum selfadjoint operator.
  In particular, the statement $(d_n) \in \Lim(X)$ implies $X = \essspec(N)$ of our \autoref{thm:diagonalizable-by-I-plus-HS} is an extension to finite spectrum normal operators of their corresponding result \cite[Theorem~4.1(ii)]{BJ-2015-TAMS} for selfadjoint operators.
  Similarly, our formula \eqref{eq:trace-in-K-spec-N} of \autoref{thm:arveson-reformulated} generalizes \cite[Theorem~4.1(iii)]{BJ-2015-TAMS}.
\end{remark}

We conclude with another perspective on the trace $\trace\big(E(N-N')\big)$.
Our next corollary shows that when the $\mathbb{Z}$-module $K_{\spec(N)}$ has full rank (i.e., $\rank K_{\spec(N)}$ is one less than the number of elements in the spectrum), this trace is zero if and only if $N'$ is a diagonalization of $N$ by a unitary $U = I + K$ with $K$ Hilbert--Schmidt.

\begin{corollary}
  Suppose $N$ is a normal operator with $\spec(N) = \{\lambda_1,\ldots,\lambda_m\}$ such that $\lambda_1-\lambda_2,\ldots,\lambda_1-\lambda_m$ are linearly independent in the $\mathbb{Z}$-module $K_{\spec(N)}$.
  Suppose further that $N$ is diagonalizable by a unitary which is a Hilbert--Schmidt perturbation of the identity.
  If $N'$ is a diagonal operator such that $E(N-N')$ is trace-class and $\trace\big(E(N-N')\big) = 0$, then there is a unitary $U = I+K$ with $K$ Hilbert--Schmidt such that $UNU^{*} = N'$.
\end{corollary}

\begin{proof}
  By \autoref{thm:arveson-reformulated}, the differences $P_k - Q_k$ are Hilbert--Schmidt and
  \begin{equation*}
    0 = \trace\big(E(N-N')\big) = \sum_{k=1}^m [P_k:Q_k]\lambda_k
  \end{equation*}
  Since $\sum_{k=1}^m [P_k:Q_k] = 0$, we have $[P_1:Q_1] = - \sum_{k=2}^m [P_k:Q_k]$ and so we may rearrange the equality above to
  \begin{equation*}
    0 = \sum_{k=2}^m [P_k:Q_k](\lambda_1 - \lambda_k).
  \end{equation*}
  Since $\lambda_1-\lambda_2,\ldots,\lambda_1-\lambda_m$ are linearly independent in $K_{\spec(N)}$, we conclude that the coefficients $[P_k:Q_k] = 0$ for $2 \le k \le m$.
  In turn, this implies $[P_1:Q_1] = 0$.
  Therefore, by \autoref{lem:restricted-conjugation-of-sets-of-projections}, there is a unitary $U = I+K$ with $K$ Hilbert--Schmidt conjugating each $P_k$ to $Q_k$.
  Thus $UNU^{*} = N'$.
\end{proof}

\emergencystretch=3em
\printbibliography

\end{document}

%% file: loreaux3.tex
\newcommand{\abs}[1]{\ensuremath{\left\lvert #1 \right\rvert}}
\newcommand{\angles}[1]{\ensuremath{\left\langle #1 \right\rangle}}

\newcommand{\sangles}[1]{\ensuremath{\langle #1 \rangle}}

\expandafter\mathchardef\expandafter\varphi\number\expandafter\phi\expandafter\relax
\expandafter\mathchardef\expandafter\phi\number\varphi